\documentclass[reqno]{amsart}
\usepackage{amsmath,amssymb,amsfonts}
   \DeclareMathOperator{\Id}{Id}
   \DeclareMathOperator{\e}{e}
   \DeclareMathOperator{\Lip}{Lip}
\usepackage{enumitem}
   \newcommand{\lb}{label}
   \newcommand{\lm}{leftmargin}

\usepackage{euscript}
\usepackage{dsfont}
   \newcommand{\N}{\ensuremath{\mathds N}}
   \newcommand{\R}{\ensuremath{\mathds R}}
\newtheorem{theorem}{Theorem}
\newtheorem{lemma}{Lemma}

\newtheorem{example}{Example}

\renewcommand{\phi}{\varphi}
\newcommand{\eps}{\varepsilon}
\newcommand{\lbd}{\lambda}
\renewcommand{\ge}{\geqslant}

\renewcommand{\le}{\leqslant}

\newcommand{\prts}[1]{\left(#1\right)}
\newcommand{\prtsr}[1]{\left[#1\right]}

\newcommand{\normn}[1]{\left\|#1\right\|_n}
\newcommand{\set}[1]{\left\{#1\right\}}
\newcommand{\setm}[1]{\setminus\set{#1}}
\newcommand{\maxs}[1]{\max\set{#1}}
\newcommand{\pfrac}[2]{\prts{\dfrac{#1}{#2}}}
\newcommand{\dsum}{\displaystyle\sum}
\def\cA{\EuScript{A}}
\def\cB{\EuScript{B}}
\def\cF{\EuScript{F}}
\def\cV{\EuScript{V}}
\def\cX{\EuScript{X}}
\begin{document}
\title
   [Nonuniform dichotomic behavior: Lipschitz invariant...]
   {Nonuniform dichotomic behavior: Lipschitz invariant manifolds for difference equations}
\author[Ant\'onio J. G. Bento]{Ant\'onio J. G. Bento}
\address{Ant\'onio J. G. Bento\\
   Departamento de Matem\'atica\\
   Universidade da Beira Interior\\
   6201-001 Covilh\~a\\
   Portugal}
\email{bento@ubi.pt}
\author{C\'esar M. Silva}
\address{C\'esar M. Silva\\
   Departamento de Matem\'atica\\
   Universidade da Beira Interior\\
   6201-001 Covilh\~a\\
   Portugal}
\email{csilva@ubi.pt}
\urladdr{www.mat.ubi.pt/~csilva}
\date{\today}
\begin{abstract}
   We obtain global and local theorems on the existence of invariant manifolds for perturbations of non autonomous linear difference equations assuming a very general form of dichotomic behavior for the linear equation. The results obtained include situations where the behavior is far from hyperbolic. We also give several new examples and show that our result includes as particular cases several previous theorems.
\end{abstract}
\maketitle
\section{Introduction}

The study of invariant manifolds is a key subject in the qualitative theory of dynamical systems. Usually, hyperbolicity is the tool used to establish the existence of stable, unstable and central invariant manifolds for perturbations of linear systems. In the study of difference and differential equations, hyperbolicity is seldom given by the existence of an (uniform) exponential dichotomy, notion that goes back to the seminal work of Perron~\cite{Perron-MZ-1930}.

In some situations, particularly in the nonautonomous setting, the concept of uniform exponential dichotomy is too restrictive and it is important to look for more general hyperbolic behavior. Two different perspectives can be identified as ways to generalize the concept of uniform exponential dichotomy: on the one hand one can consider growth rates that are not necessarily exponential and on the other hand one can define dichotomies that depend on the initial time and therefore are nonuniform.
The first approach is present in the work of Pinto~\cite{Pinto-CMA-1994, Pinto-PFICDE-1995} and Naulin and Pinto~\cite{Naulin-Pinto-NATMA-1994} where the authors study stability of ordinary differential linear equations possessing $(h,k)$-dichotomies and notion is closely related with the notion of uniform dichotomy used in the book by P\"otzsche~\cite{Potzsche-LNM-2010}. The second approach lead to concepts of nonuniform exponential dichotomies and can be found in the work of Preda and Megan~\cite{Preda-Megan-BAusMS-1983} and Megan, Sasu and Sasu~\cite{Megan-Sasu-Sasu-IEOT-2002}, and in a different form, in the work of Barreira and Valls~\cite{Barreira-Valls-DCDS-A-2006}. The last authors studied deeply, in the general framework of Banach spaces, the existence of stable and central manifolds for nonlinear and nonautonomous perturbations of linear nonautonomous difference and differential equations assuming that the linear equation admits a nonuniform exponential dichotomy.

A natural step towards generalization is to obtain invariant manifolds for dichotomies that are both nonuniform and not necessarily exponential. This was the approach followed by the present authors that in~\cite{Bento-Silva-JFA-2009,Bento-Silva-QJM-2012} obtained stable manifolds for nonautonomous nonlinear perturbations of nonautonomous linear difference and differential equations, assuming the existence of a nonuniform polynomial dichotomy for the linear equation, and in~\cite{Bento-Silva-arXiv:1007.5039v1,Bento-Silva-NATMA-2012} where it was assumed that the linear equation belongs to a more general family of nonuniform dichotomies - the so called $(\mu,\nu)$-dichotomies. We note that the concept of nonuniform polynomial dichotomy considered in~\cite{Bento-Silva-JFA-2009} is not included as a particular case in the concept of $(\mu,\nu)$-dichotomy, though, letting $\mu_m=\nu_m=m$ in~\eqref{eq:mu-nu-dich} (see Example~\ref{example:mu-nu}), we can obtain a different version of polynomial dichotomy. This last version coincides with the notion of dichotomy considered by Barreira and Valls~\cite{Barreira-Valls-NATMA-2009}, where some conditions for the existence of polynomial behavior are obtained in terms of generalized polynomial Lyapunov exponents. We emphasize that our $(\mu,\nu)$-dichotomies allowed the obtention of stable manifolds in situations where the classical Lyapunov exponents are zero, for instance for the referred nonuniform polynomial dichotomies. It should also be mentioned that Barreira and Valls~\cite{Barreira-Valls-N-2010} were able to obtain stable manifolds for perturbations of linear equations assuming that the linear equation has some type of dichotomy given by general growth rates that correspond to set $\mu_m=\nu_m=\e^{\rho(m)}$ in~\eqref{eq:mu-nu-dich} with $\rho:\N \to \N$ increasing. Naturally, the last condition implies that the rates growth exponentially or faster, and do not include zero Lyapunov exponent situations like, for example, the referred polynomial behavior.

For a different approach to the existence of invariant manifolds for nonautonomous difference equations we recommend again the book by P\"otzsche~\cite{Potzsche-LNM-2010} that includes also historical notes about the development of the subject as well as a large set of references.

The several concepts of nonuniform dichotomy that have been considered in recent years led us to consider the following tasks: define a general framework that includes as particular cases the several definitions of nonuniform dichotomy and that still allows us to obtain results (existence and regularity of invariant manifolds, robustness results) that generalize the ones in the literature. This made us consider some type of general dichotomic behavior that consists simply in assuming the existence of a splitting into two sequences of invariant subspaces where the norms of the evolution map are bounded by some double sequences that depend on the initial and final times (see~\ref{eq:split4} and~\ref{eq:split5}). In this paper we establish the existence on Lipschitz invariant manifolds for perturbations of nonautonomous linear equations with the mentioned dichotomic behavior, obtaining an asymptotic behavior along the manifolds that is the same as the one assumed for the linear part in the corresponding subspaces.

Firstly we obtain a global theorem on the existence of invariant manifolds for perturbations of linear difference equations and as a consequence of this result we obtain a local result for a large set of perturbations. The perturbations considered are nonautonomous in the sense that we perturb the diference equation by a different Lipschitz function $f_n$ for each time $n$. Several examples illustrate our results and we also check that our result include, as particular cases, several previous theorems. We stress that our approach reveals the relation between the Lipschitz constants of the perturbations and the behavior assumed for the dichotomy. This allows us to require less from the dichotomic behavior by restricting the set of perturbations or to consider an extended set of perturbations by demanding more from the dichotomic behavior. We also stress that we include in our theorems situations that are far from being hyperbolic in any reasonable sense.

\section{Notation and preliminaries}
Let $\N_\ge^2=\set{\prts{m,n} \in \N^2 \colon m \ge n}$ and $\N_>^2 = \set{\prts{m,n} \in \N^2 \colon m > n}$. Let $B(X)$ be the space of bounded linear operators in a Banach space~$X$. Given a sequence $(A_n)_{n\in\N}$ of operators of $B(X)$, we write
   $$ \cA_{m,n} =
      \begin{cases}
         A_{m-1} \cdots A_n & \text{ if } m>n, \\
         \text{Id} & \text{ if } m=n,
      \end{cases}$$
for every $\prts{m,n} \in \N_\ge^2$.

Consider the linear difference equation
\begin{equation} \label{eq:lin:dif}
   x_{m+1} = A_m x_m, \ m \in \N.
\end{equation}

We say that equation~\eqref{eq:lin:dif}
admits an \textit{invariant splitting} if there exist bounded projections $P_n$, $n\in\N$, such that, for every $(m,n) \in \N_\ge^2$, we have
\begin{enumerate}[\lb=$($S$\arabic*)$,\lm=13mm]
   \item $ P_m \cA_{m,n} = \cA_{m,n} P_n$; 
   \item $\cA_{m,n}(\ker P_n) = \ker P_m$; 
   \item $\cA_{m,n}|_{\ker P_n} \colon \ker P_n \to \ker P_m$ is invertible. \end{enumerate}
In these conditions we define, for each $n \in \N$, the complementary projection $Q_n=\Id-P_n$ and the linear subspaces $E_n=P_n (X)$ and $F_n= \ker P_n = Q_n(X)$. As usual, we identify the vector spaces $E_n \times F_n$ and $E_n \oplus F_n$ as the same vector space.

Given double sequences $(a_{m,n})_{(m,n) \in \N_\ge^2}$ and $(b_{m,n})_{(m,n) \in \N_\ge^2}$ we say that equation~\eqref{eq:lin:dif} admits a \textit{general dichotomy with bounds $\prts{a_{m,n}}$ and $\prts{b_{m,n}}$} if it admits an invariant splitting such that
\begin{enumerate}[\lb=$($D$\arabic*)$,\lm=13mm]
   \item $\|\cA_{m,n} P_n\| \le a_{m,n}$; \label{eq:split4}
   \item $\|(\cA_{m,n}|_{F_n})^{-1}Q_m\| \le b_{m,n}$. \label{eq:split5}
\end{enumerate}

\begin{example}
Let $(a_n)_{n \in \N},(b_n)_{n \in \N},(c_n)_{n \in \N}$ and $(d_n)_{n \in \N}$ be given sequences of positive numbers such that $c_n \ge 1$ and $d_n \ge 1$ for every $n \in \N$. Consider the linear operators $A_n : \R^2 \to \R^2$ given by the diagonal matrices
      $$ A_n = \prtsr{
         \begin{array}{cc}
            \dfrac{a_n}{a_{n+1}} \, \
               \pfrac{c_n^{1-(-1)^n}}
               {c_{n+1}^{1+(-1)^n}}^{1/2}
            & 0 \\
            0 &
               \dfrac{b_{n+1}}{b_n} \, \
               \pfrac{d_n^{1-(-1)^n}}
               {d_{n+1}^{1+(-1)^n}}^{1/2}
         \end{array}}.$$
Considering the projections given by $P_n(x,y) = (x,0)$ and $Q_n(x,y) = (0,y)$ we have
      $$ \|\cA_{m,n}P_n\|
         = \dfrac{a_n}{a_m} \,
            \pfrac{c_n^{1-(-1)^n}}{c_m^{1+(-1)^m}}^{1/2}$$
   and
      $$ \|(\cA_{m,n}|_{F_n})^{-1}Q_m\|
         = \dfrac{b_n}{b_m} \,
            \pfrac
            {d_m^{1+(-1)^m}}{d_n^{1-(-1)^n}}^{1/2}$$
   and this implies
      \begin{equation*}
      \|\cA_{m,n}P_n\|
         \le \dfrac{a_n}{a_m} \, c_n
         \ \ \ \text{ and } \ \ \
         \|(\cA_{m,n}|_{F_n})^{-1}Q_m\|
         \le \dfrac{b_n}{b_m} \, d_m.
       \end{equation*}
   This example shows that for any given sequences $(a_n)_{n \in \N}$, $(b_n)_{n \in \N}$, $(c_n)_{n \in \N}$ and $(d_n)_{n \in \N}$ of positive numbers such that $c_n \ge 1$ and $d_n \ge 1$ for every $n \in \N$, there is always a sequence of linear maps $(A_n)_{n \in \N}$ that admits a general dichotomy with bounds
      $$ \left(\frac{a_n}{a_m}c_n\right)_{(m,n) \in \N_\ge^2}
         \quad \text{and} \quad
         \left(\frac{b_n}{b_m}d_m\right)_{(m,n) \in \N_\ge^2}$$
   for the projections above.
   
   In particular, if
      $$ a_m = \e^{am}, \quad
         b_m = \e^{-bm} \quad \text{ and } \quad
         c_m = d_m = D \e^{\eps m},$$
   for some constants $a < 0 \le b$ and $\eps>0$, we obtain an example of a linear equation that admits a nonuniform exponential dichotomy (in the sense of~\cite{Barreira-Valls-DCDS-A-2006}).
   
   Another particular case can be obtained by setting
      $$ a_m=m^a, \quad
         b_m=m^{-b} \quad \text{ and } \quad
         c_m=d_m= Dm^\eps,$$
   for some constants $a < 0 \le b$ and $\eps>0$. Here we obtain an example of a linear equation that admits a nonuniform polynomial dichotomy like the ones considered in~\cite{Barreira-Fan-Valls-Zhang-TMNA-2011}.
\end{example}
\section{Existence of invariant Lipschitz manifolds}
In this section we are going to state our results on the existence of Lipschitz invariant manifolds of the difference equation
   $$ x_{m+1} = A_m x_m + f_m(x_m), \ m \in \N,$$
where $f_m : X \to X$ are Lipschitz perturbations such that
\begin{equation}
   f_m(0)=0 \text{ for every $m \in \N$} \label{cond-f-0}.
\end{equation}
For every $f_m$ we define
\begin{equation}\label{eq:lip_const}
   \Lip(f_m) = \sup\set{\dfrac{\|f_m(u) - f_m(v)\|}{\|u - v\|}
      \colon u, v \in X,\ u \ne v}.
\end{equation}

Given $n\in \N$ and $v_n=(\xi,\eta)\in E_n \times F_n$, for each $m>n$ we write
\begin{equation}\label{eq:traj}
   v_m
   =\cF_{m,n}(v_n)
   = \cF_{m,n}(\xi,\eta)
   =(x_m,y_m) \in E_m \times F_m,
\end{equation}
with
\begin{equation} \label{eq:dyn}
   \cF_{m,n} =
      \begin{cases}
         (A_{m-1}+f_{m-1})\circ \cdots\circ (A_n+f_n) & \text{ if } m>n, \\
         \text{Id} & \text{ if } m=n.
      \end{cases}
\end{equation}

We denote by $\cX$ the space of sequences $(\phi_n)_{n \in \N}$ of functions $\phi_n\colon E_n \to F_n$ such that
\begin{align}
   & \phi_n(0)=0 \label{cond-phi-0} \\
   & \| \phi_n(\xi) - \phi_n(\bar\xi) \| \le \| \xi - \bar\xi\|
   \label{cond-phi-1}
\end{align}
for every $\xi$, $\bar{\xi} \in E_n$ and every $n \in \N$. Note that making $\bar\xi = 0$ in~\eqref{cond-phi-1} we have
\begin{equation} \label{cond-phi-1a}
   \|\phi_n(\xi)\| \le \|\xi\|
\end{equation}
for every $n \in \N$ and every $\xi \in E_n$.

Given $(\phi_n)_{n \in \N} \in \cX$, for each $n \in \N$, we consider the graph
\begin{equation}\label{def:V_phi,n}
   \cV_{\phi,n}
   = \set{(\xi,\phi_n(\xi)): \xi \in E_n},
\end{equation}
that we call \textit{global invariant manifold}.

We now state the result on the existence of global invariant manifolds.

\begin{theorem} \label{thm:global}
   Given a Banach space $X$, suppose that equation~\eqref{eq:lin:dif} admits a general dichotomy with bounds $(a_{m,n})_{(m,n) \in \N_\ge^2}$ and $(b_{m,n})_{(m,n) \in \N_\ge^2}$. Let $f_m : X \to X$ be a sequence of Lipschitz functions satisfying \eqref{cond-f-0}. Assume that
   \begin{equation} \label{eq:CondicaoTeo}
      \lim_{m \to +\infty} a_{m,n} b_{m,n} = 0
   \end{equation}
   for every $n \in \N$,
   \begin{equation}\label{eq:alpha}
      \alpha =
      \sup_{(m,n) \in \N_>^2}
      \dfrac{1}{a_{m,n}} \dsum_{k=n}^{m-1} a_{m,k+1} a_{k,n} \Lip(f_k)
      < + \infty
   \end{equation}
   and
   \begin{equation}\label{eq:beta}
      \beta =
      \sup_{n \in \N} \dsum_{k=n}^{+\infty} b_{k+1,n} a_{k,n} \Lip(f_k)
      < \infty.
   \end{equation}
   If
   \begin{equation}\label{ine:alpha_beta}
      2 \alpha + \max\set{2 \beta, \sqrt{\beta}} < 1,
   \end{equation}
   then there is a unique $\phi \in \cX$ such that
   \begin{equation} \label{thm:global:invar}
      \cF_{m,n}(\cV_{\phi,n})
      = \cV_{\phi,m} \
      \text{for every } \prts{m,n} \in \N_\ge^2,
   \end{equation}
   where $\cV_{\phi,n}$ and $\cV_{\phi,m}$ are given by~\eqref{def:V_phi,n}.
   Furthermore, we have
   \begin{equation}\label{thm:ineq:norm:F_mn(xi...)-F_mn(barxi...)}
      \|\cF_{m,n}(\xi,\phi_n(\xi)) -\cF_{m,n}(\bar \xi,\phi_n(\bar \xi)) \|
      \le \dfrac{2}{1-2\alpha} a_{m,n}\, \|\xi-\bar \xi\|.
   \end{equation}
   for every $\prts{m,n} \in \N_\ge^2$ and every $\xi, \bar\xi \in E_n$.
\end{theorem}

We will now consider the problem of existence of local invariant manifolds.

Let $B(r)$ denote the open ball of radius $r$ in $X$ and define
\begin{equation}\label{eq:manifold-local}
\cV^*_{\phi,n,r}=\set{(\xi,\phi_n(\xi)) \in \cV_{\phi,n} \colon \xi \in B(r)}.
\end{equation}

The next theorem can be obtained from Theorem~\ref{thm:global}.

\begin{theorem} \label{thm:local}
   Given a Banach space $X$, suppose that equation~\eqref{eq:lin:dif} admits a general dichotomy with bounds $(a_{m,n})_{(m,n)\in \N_\ge^2}$ and $(b_{m,n})_{(m,n)\in \N_\ge^2}$. For each $m \in \N$, let $f_m : X \to X$ be a Lipschitz function in $B(r_m)$  satisfying \eqref{cond-f-0}. Assume that
   \begin{equation} \label{eq:CondicaoTeo-local}
      \lim_{m \to +\infty} a_{m,n} b_{m,n} = 0
   \end{equation}
   for every $n \in \N$,
   \begin{equation}\label{eq:tau-local}
      \alpha =
      \sup_{(m,n) \in \N_>^2}
      \dfrac{1}{a_{m,n}} \dsum_{k=n}^{m-1} a_{m,k+1} a_{k,n}
      \Lip(f_k|_{B(r_k)}) < + \infty
   \end{equation}
   and
   \begin{equation}\label{eq:lbd-local}
      \beta =
      \sup_{n \in \N} \dsum_{k=n}^{+\infty} b_{k+1,n} a_{k,n}
      \Lip(f_k|_{B(r_k)}) < + \infty.
   \end{equation}
   If, for each $n \in \N$,
   \begin{equation}\label{ine:sn-local}
      s_n= \max\left\{ 1, \, \frac{2}{1-4\alpha} \, \sup_{m \ge n} \frac{a_{m,n} r_n}{r_m} \right\}< +\infty
   \end{equation}
   and
   \begin{equation}\label{ine:alpha_beta-local}
      4 \alpha + \max\set{4 \beta, \sqrt{2\beta}} < 1,
   \end{equation}
   then there is $\phi \in \cX$ such that
   \begin{equation} \label{thm:global:invar-local}
      \cF_{m,n}(\cV^*_{\phi,n,r_n/(2s_n)})
      \subseteq \cV^*_{\phi,m,r_m} \
      \text{for every } \prts{m,n} \in \N_\ge^2.
   \end{equation}
   
   Furthermore, we have
   \begin{equation}\label{thm:ineq:norm:F_mn(xi...)-F_mn(barxi...)-local}
      \|\cF_{m,n}(\xi,\phi_n(\xi)) -\cF_{m,n}(\bar \xi,\phi_n(\bar \xi)) \|
      \le \dfrac{2}{1-4\alpha} \, a_{m,n}\, \|\xi-\bar \xi\|.
   \end{equation}
   for every $\prts{m,n} \in \N_\ge^2$ and every $\xi, \bar\xi \in B(r_n)$.
\end{theorem}
\section{Examples} 
In this section we will give some examples that illustrate our theorem and show that it contains as a particular case several results in the literature.

Firstly, we present some results on the existence of global invariant manifolds.

\begin{example}
   For each $(m,n) \in \N_\ge^2$, set
      $$ a_{m,n}= \dfrac{a_n}{a_m} c_n \quad \text{ and } \quad
         b_{m,n}=\dfrac{b_n}{b_m} d_m,$$
   where $(a_m)_{m \in \N}$, $(b_m)_{m \in \N}$, $(c_m)_{m \in \N}$ and $(d_m)_{m \in \N}$ are some nondecreasing sequences of positive numbers. In this particular case, conditions~\eqref{eq:CondicaoTeo}, \eqref{eq:alpha} and \eqref{eq:beta} correspond respectively to the conditions
	\begin{equation} \label{eq:condddd}
	  \lim_{m \to +\infty} \dfrac{d_m}{a_m b_m} = 0,
	\end{equation}
	  $$\alpha=\sum_{k=1}^{+\infty} \dfrac{a_{k+1}}{a_k} c_{k+1}
	\Lip(f_k) < +\infty$$
   and
 	 $$\beta=\sup_{n \in \N} b_n a_n c_n \sum_{k=n}^{+\infty}
 	 \dfrac{d_{k+1}}{a_k b_{k+1}} \Lip(f_k) < +\infty.$$
   Thus, if the numbers $\Lip(f_k)$ are small enough so that~\eqref{ine:alpha_beta} holds, we obtain a sequence of invariant manifolds $\cV_{\phi,n}$ given by~\eqref{def:V_phi,n} where the decay is given by
      $$ \|\cF_{m,n}(\xi,\phi_n(\xi)) -\cF_{m,n}(\bar \xi,\phi_n(\bar \xi)) \|
         \le \dfrac{2}{1-2\alpha} \, \dfrac{a_n}{a_m} c_n\, \|\xi-\bar \xi\|,$$
   for every $\prts{m,n} \in \N_\ge^2$ and every $\xi, \bar\xi \in E_n$.

   It is easy to see that, if we set
      $$ \Lip(f_k)
         \le \maxs{\dfrac{a_k}{a_{k+1}c_{k+1}},
            \dfrac{a_k b_{k+1}}{d_{k+1} \max\limits_{1\le i \le k} \prts{a_ib_ic_i}}}
            \lbd_k \quad \text{with} \quad \lambda=\sum_{k=1}^{+\infty} \lambda_k < \frac{1}{4},$$
   conditions~\eqref{eq:alpha}, \eqref{eq:beta} and~\eqref{ine:alpha_beta} are verified and thus, provided that~\eqref{eq:condddd} holds, we always have perturbations with small enough non-zero Lipschitz constants such that the perturbed equations has invariant manifolds with the behavior given in our theorem.

   In particular, setting
      $$ a_m = \e^{am}, \quad
         b_m = \e^{-bm} \quad \text{ and } \quad
         c_m = d_m = D \e^{\eps m},$$
   for some constants $a < 0 \le b$ and $\eps>0$, we get
      $$ a_{m,n}=D\e^{a(m-n)+\eps n} \quad \text{and} \quad
         b_{m,n}=D\e^{b(m-n)+\eps m},$$
   and we obtain Theorem 3 in~\cite{Barreira-Valls-DCDS-A-2012}. Note that for these dichotomies condition~\eqref{eq:condddd} is equivalent to condition $a+\eps<b$, already present in the referred paper.

   Another particular case can be obtained by setting
      $$ a_m=m^a, \quad
         b_m=m^{-b} \quad \text{ and } \quad
         c_m=d_m= Dm^\eps,$$
   for some constants $a < 0 \le b$ and $\eps>0$. In this case we get
      $$ a_{m,n} = D\left( \dfrac{m}{n}\right)^a n^\eps  \quad \text{and} \quad
         b_{m,n}=D \left( \dfrac{m}{n}\right)^{-b} m^\eps,$$
   corresponding to the dichotomies already considered for instance in~\cite{Barreira-Valls-NATMA-2009,Barreira-Fan-Valls-Zhang-TMNA-2011}.
   In this case condition~\eqref{eq:condddd} is also equivalent to $a+\eps<b$ and assuming this condition we obtain global Lipschitz stable manifolds for small enough Lipschitz perturbations of linear equations admitting these polynomial dichotomies. As far as we are aware in this polynomial setting this result was obtained here for the first time.
\end{example}

\begin{example}
   For each $(m,n) \in \N_\ge^2$, set
      $$ a_{m,n} = D(m-n+1)^a n^\eps \quad \text{and} \quad
         b_{m,n} = D(m-n+1)^{-b} m^\eps,$$
   for some constants $a < 0 \le b$ and $\eps>0$. Set also
      $$ \Lip(f_k) \le \delta k^{-2\eps-1}.$$
   In this case, conditions~\eqref{eq:alpha} and \eqref{eq:beta} are satisfied, condition~\eqref{eq:CondicaoTeo} corresponds to $a+\eps<b$ and condition~\eqref{ine:alpha_beta} is satisfied if we consider a small enough $\delta>0$.
   Thus our theorem allows us to obtain a sequence of invariant manifolds $\cV_{\phi,n}$ given by~\eqref{def:V_phi,n} where the decay is given by
   $$\|\cF_{m,n}(\xi,\phi_n(\xi)) -\cF_{m,n}(\bar \xi,\phi_n(\bar \xi)) \|
   \le \dfrac{2D}{1-2\alpha} \, (m-n+1)^a n^\eps \, \|\xi-\bar \xi\|,$$
   for every $\prts{m,n} \in \N_\ge^2$ and every $\xi, \bar\xi \in E_n$.

   This result corresponds to Theorem~2 in~\cite{Bento-Silva-JFA-2009}.
\end{example}

\begin{example}
   For each $(m,n) \in \N_\ge^2$, set
      $$ a_{m,n} = D\e^{a(m-n)+\eps n} \quad \text{and} \quad
         b_{m,n}=D \prts{\frac{m}{n}}^{-b} m^\eps,$$
   for some constants $a < 0 \le b$ and $\eps>0$. Set also
      $$ \Lip(f_k) \le \delta k^{-2\eps-1}.$$
   In this case, all conditions of our theorem are satisfied provided that we consider a small enough $\delta>0$.
   Thus our theorem allows us to obtain a sequence of invariant manifolds $\cV_{\phi,n}$ given by~\eqref{def:V_phi,n} where the decay is given by
      $$ \|\cF_{m,n}(\xi,\phi_n(\xi)) -\cF_{m,n}(\bar \xi,\phi_n(\bar \xi)) \|
         \le \dfrac{2D}{1-2\alpha} \, \e^{a(m-n)+\eps n} \, \|\xi-\bar \xi\|,$$
   for every $\prts{m,n} \in \N_\ge^2$ and every $\xi, \bar\xi \in E_n$.
\end{example}

\begin{example}
   For each $(m,n) \in \N_\ge^2$, set
      $$ a_{m,n} = L \quad \text{and} \quad
         b_{m,n}=D \e^{a(m-n)+\eps m},$$
   for some constants $L \ge 1$, $a<0$ and $\eps>0$. Set also $\Lip(f_k)=\delta \e^{-\eps k}$. Once again all conditions of our theorem are satisfied provided that we consider a small enough $\delta>0$.
   Thus our theorem allows us to obtain a sequence of invariant manifolds $\cV_{\phi,n}$ given by~\eqref{def:V_phi,n} where we have
   $$\|\cF_{m,n}(\xi,\phi_n(\xi)) -\cF_{m,n}(\bar \xi,\phi_n(\bar \xi)) \|
   \le \dfrac{2L}{1-2\alpha} \, \|\xi-\bar \xi\|,$$
   for every $\prts{m,n} \in \N_\ge^2$ and every $\xi, \bar\xi \in E_n$. That is, we obtain an upper bound for the distance of the iterates of any two points in the manifolds.

   In particular, setting for each $n \in \N$,
   $$A_n=
   \left[
   \begin{array}{cc}
   L^{(-1)^n} & 0 \\
   0 & 1/2
   \end{array}
   \right]
   $$
   we obtain $b_{m,n}=2^{-(m-n)}$ and $a_{m,n} \in \{1/L,1,L\}$ (and therefore $a_{m,n} \le L$). This shows that the given sequence of matrices satisfies the hypothesis above. This example shows that we can still obtain some informations for the dynamics in situations that are far from being hyperbolic in any reasonable sense.
\end{example}

The next examples are special cases of theorem~\ref{thm:local}.

\begin{example}
   For each $(m,n) \in \N_\ge^2$, set $a_{m,n}= \dfrac{a_n}{a_m} c_n$ and $b_{m,n}=\dfrac{b_n}{b_m} d_m$ where $(a_m)_{m \in \N}$, $(b_m)_{m \in \N}$, $(c_m)_{m \in \N}$ and $(d_m)_{m \in \N}$ are some sequences of positive numbers and assume that, for each $k \in \N$, we have
	  $$ \|f_k(u)-f_k(v)\| \le c \|u-v\| (\|u\|+\|v\|)^q,$$
   for some constants $c>0$ and $q \ge 0$. It is immediate that $f_k|_{B(r_k)}$ is Lipschitz with Lipschitz constant less or equal to $c 2^q r_k^q$. Thus, conditions~\eqref{eq:CondicaoTeo-local}, \eqref{eq:tau-local}, \eqref{eq:lbd-local} and~\eqref{ine:sn-local} correspond respectively to the conditions
	\begin{equation*} 
   	\lim_{m \to +\infty} \dfrac{d_m}{a_m b_m} = 0,
	\end{equation*}
	\begin{equation*} 
	  \alpha = 
      c 2^q \sum_{k=1}^{+\infty} \dfrac{a_{k+1}}{a_k} c_{k+1} r_k^q
      < +\infty,
	\end{equation*}
	\begin{equation*} 
      \beta = c 2^q 
      \sup_{n \in \N} a_n b_n c_n 
      \sum_{k=n}^{+\infty} \dfrac{d_{k+1} r_k^q}{a_k b_{k+1}} < +\infty
 	\end{equation*}
   and
	\begin{equation*} 
      a_n c_n r_n \sup_{m \ge n} a_m^{-1} r_m^{-1} < +\infty.
	\end{equation*}
   Thus, if the radius of the balls $B(r_k)$ are small enough so that~\eqref{ine:alpha_beta-local} holds, we obtain a sequence of invariant manifolds given by~\eqref{eq:manifold-local} where the decay is given by
      $$ \|\cF_{m,n}(\xi,\phi_n(\xi)) -\cF_{m,n}(\bar \xi,\phi_n(\bar \xi)) \|
         \le \dfrac{2}{1-4\alpha} \, \dfrac{a_n}{a_m} c_n\, \|\xi-\bar \xi\|,$$
   for every $\prts{m,n} \in \N_\ge^2$ and every $\xi, \bar\xi \in E_n$. Note that, for any dichotomy in this example, it is always possible to choose small enough balls where our hypothesis hold.

   As a particular case, given $a < 0 \le b$ and $\eps > 0$, we can put
      $$ a_m= \e^{- a m}, \quad
         b_m= \e^{b m} \quad \text{ and }
         c_m= d_m=D \e^{\eps m},$$
   for each $m \in \N$. We can also set $r_k=\delta \e^{-\beta k}$ for each $k \in \N$. In this setting condition~\eqref{eq:CondicaoTeo-local} is equivalent to $a+\eps<b$, conditions~\eqref{eq:tau-local} and \eqref{eq:lbd-local} are equivalent to $\eps-\beta q<0$ and condition~\eqref{ine:sn-local} is equivalent to $a+\beta \le 0$ (this condition implies $a+\eps<b$).  With this setting we obtain the result in~\cite{Barreira-Valls-DCDS-A-2006}. In fact we slightly improve that result since in our case we can have $\beta>\frac{2\eps}{q}$ while $\beta=\eps+\frac{2\eps}{q}$ in~\cite{Barreira-Valls-DCDS-A-2006}.
\end{example}

\begin{example} \label{example:mu-nu}
   Given $a < 0 \le b$ and $\eps > 0$, for each $(m,n) \in \N_\ge^2$, set
   \begin{equation} \label{eq:mu-nu-dich}
      a_{m,n}= D \prts{\dfrac{\mu_m}{\mu_{n-1}}}^a \nu_{n-1}^\eps
      \quad \text{and} \quad
      b_{m,n}=D \prts{\dfrac{\mu_{m-1}}{\mu_n}}^{-b} \nu_{m-1}^\eps
   \end{equation}
   where $(\mu_m)_{m \in \N_0}$ and $(\nu_m)_{m \in \N_0}$ are growth rates, that is these sequences are non decreasing, converge to $+\infty$ and $\mu_0=\nu_0=1$. Also assume, for each $k \in \N$, that $r_k=\delta R_k$ and that we have
   	$$ \|f_k(u)-f_k(v)\| \le c \|u-v\| (\|u\|+\|v\|)^q,$$
   for some constants $c>0$ and $q \ge 0$.
   In this case, conditions~\eqref{eq:CondicaoTeo-local}, \eqref{eq:tau-local}, and~\eqref{ine:sn-local} correspond respectively to the conditions
	\begin{equation*}
   	\lim_{m \to +\infty} \dfrac{\mu_m^a \nu_{m-1}^\eps}{\mu_{m-1}^b} = 0,
	\end{equation*}
	\begin{equation*}
	  \sum_{k=1}^{+\infty} \nu_k^\eps R_k^q < +\infty,
	\end{equation*}
	and
	\begin{equation*}
      D \frac{R_n \nu_{n-1}^\eps}{\mu_{n-1}}^a \sup_{m \ge n}
   	\frac{\mu_m^a}{R_m} < +\infty.
	\end{equation*}
	Additionally
	\begin{equation*}
      \sup_{n \in \N} \nu_{n-1}^\eps \sum_{k=n}^{+\infty} \nu_k^\eps R_k^q
      < +\infty
 	\end{equation*}
   implies \eqref{eq:lbd-local}.

   Thus, if the provided $c>0$ is small enough so that~\eqref{ine:alpha_beta-local} holds, we obtain a sequence of invariant manifolds given by~\eqref{eq:manifold-local} where the decay is given by
      $$ \|\cF_{m,n}(\xi,\phi_n(\xi)) -\cF_{m,n}(\bar \xi,\phi_n(\bar \xi))\|
         \le \dfrac{2D}{1-4\alpha} \, \prts{\dfrac{\mu_m}{\mu_{n-1}}}^a
            \nu_{n-1}^\eps \, \|\xi-\bar \xi\|,$$
   for every $\prts{m,n} \in \N_\ge^2$ and every $\xi, \bar\xi \in E_n$. Letting now $R_k=\mu_k^a$ we obtain Theorem 1 in~\cite{Bento-Silva-NATMA-2012}. In fact, in this case, condition~\eqref{eq:CondicaoTeo-local} is equivalent to (13) in~\cite{Bento-Silva-NATMA-2012}, condition~\eqref{eq:tau-local} is equivalent to (12) in~\cite{Bento-Silva-NATMA-2012}, condition~\eqref{eq:lbd-local} is implied by (15) in~\cite{Bento-Silva-NATMA-2012} and~\eqref{ine:sn-local} is immediate.
\end{example}

\begin{example}
   For each $(m,n) \in \N_\ge^2$, set
      $$ a_{m,n} = D(m-n+1)^a n^\eps \quad \text{and} \quad
         b_{m,n} = D(m-n+1)^{-b} m^\eps,$$
   for some constants $a < 0 \le b$ and $\eps>0$. Assume further that
   for each $k \in \N$ we have
   	$$ \|f_k(u)-f_k(v)\| \le c \|u-v\| (\|u\|+\|v\|)^q,$$
   for some $c>0$ and $q>0$ and that $r_k=\delta k^{-\gamma}$. In this case, conditions~\eqref{eq:tau-local} and \eqref{eq:lbd-local} are satisfied provided $\gamma>\frac{2\eps+1}{q}$, condition~\eqref{ine:sn-local} is satisfied if $\gamma \le -a$, condition~\eqref{eq:CondicaoTeo-local} corresponds to $a+\eps<b$ and condition~\eqref{ine:alpha_beta} is satisfied if we consider a small enough $\delta>0$.
   The result obtained corresponds to Theorem~1 in~\cite{Bento-Silva-JFA-2009}.
\end{example}
\section{Proof of Theorem~\ref{thm:global}}
In this section we will prove Theorem~\ref{thm:local}.
Given $n\in \N$ and $v_n=(\xi,\eta)\in E_n \times F_n$, using~\eqref{eq:traj}, it follows that for each $m > n$, the trajectory $\prts{v_m}_{m > n}$
satisfies the following equations
\begin{align}
   x_m & = \cA_{m,n} \xi + \sum_{k=n}^{m-1} \cA_{m,k+1} P_{k+1} f_k(x_k,y_k),
      \label{eq:dyn-split1a}\\
   y_m  &= \cA_{m,n} \eta +\sum_{k=n}^{m-1} \cA_{m,k+1} Q_{k+1} f_k(x_k,y_k).
      \label{eq:dyn-split1b}
\end{align}
In view of the forward invariance required in \eqref{thm:global:invar}, each trajectory of~\eqref{eq:dyn} starting in $\cV_{\phi,n}$ must be
in $\cV_{\phi,m}$ for every $\prts{m,n} \in \N_\ge^2$, and thus the equations~\eqref{eq:dyn-split1a} and~\eqref{eq:dyn-split1b} can be written in
the form
\begin{align}
   & x_m = \cA_{m,n} \xi + \sum_{k=n}^{m-1} \cA_{m,k+1} P_{k+1}
      f_k(x_k,\phi_k(x_k)), \label{eq:dyn-split2a}\\
   & \phi_m (x_m) = \cA_{m,n} \phi_n(\xi) + \sum_{k=n}^{m-1} \cA_{m,k+1}
      Q_{k+1} f_k (x_k,\phi_k(x_k)). \label{eq:dyn-split2b}
\end{align}
To prove that equations~\eqref{eq:dyn-split2a} and~\eqref{eq:dyn-split2b} have solutions we will use Banach fixed point theorem in some suitable complete metric spaces.

In $\cX$ we define a metric by
\begin{equation}\label{def:metric:X}
   d(\phi,\psi)
   = \sup\set{\dfrac{\|\phi_n(\xi) - \psi_n(\xi)\|}{\|\xi\|} : n \in \N
         \text{ and } \xi \in E_n\setminus \set{0}}.
\end{equation}
for each $\phi=(\phi_n)_{n \in \N}$, $\psi=(\psi_n)_{n \in \N} \in \cX$.
It is easy to see that $\cX$ is a complete metric space with the metric
defined by~\eqref{def:metric:X}.

Let $\cB_n$ be the space of all sequences $x=\prts{x_m}_{m\ge n}$
of functions $x_m \colon E_n \to E_m$ such that
\begin{align}
   & x_m(0) =0 \text{ for every } m \ge n, \label{cond-x_m-0}\\
   & \normn{x} = \sup\set{\dfrac{\|x_m(\xi)\|}{a_{m,n} \|\xi\|}
      \colon m \ge n, \ \xi \in E_n \setm{0}} < +\infty. \label{cond-x_m-1}
\end{align}
From~\eqref{cond-x_m-1} we obtain the following estimates
\begin{equation} \label{cond-x_m-1a}
   \|x_m(\xi)\|
   \le a_{m,n} \normn{x} \|\xi\|
\end{equation}
for every $m \ge n$ and every $\xi \in E_n$. It is easy to see that $\prts{\cB_n, \normn{\cdot}}$ is a Banach space.

\begin{lemma}\label{lemma:Exist-Suc-x_m}
   For each $\phi\in \cX$ and $n \in \N$ there exists a unique sequence $x=x^{\phi}\in \cB_n$ satisfying equation~\eqref{eq:dyn-split2a}. Moreover
   \begin{align}
      & x^\phi_n(\xi) = \xi, \label{cond-x^phi_n(xi)=xi}\\
      & \normn{x^\phi} \le \dfrac{1}{1 - 2 \alpha},\label{eq:norm_x^phi}\\
      & \label{ineq:x^phi_m(xi)- x^phi_m(barxi)<=...}
         \|x^\phi_m(\xi) - x^\phi_m(\bar\xi)\|
         \le \dfrac{1}{1-2\alpha} a_{m,n} \|\xi-\bar\xi\|
   \end{align}
   for every $m \ge n$ and $\xi$, $\bar\xi \in E_n$. Furthermore,
   \begin{equation} \label{ineq:d_n(x^phi-x^psi)<=d(phi,psi)}
      \normn{x^\phi - x^\psi}
      \le \dfrac{\alpha}{\prts{1-2\alpha}^2} d(\phi,\psi)
   \end{equation}
   for each $\phi, \psi \in \cX$.
\end{lemma}

\begin{proof}
   Given $\phi \in \cX$, we define an operator $J = J_\phi$ in~$\cB_n$ by
   \begin{equation} \label{def:J}
      (Jx)_m(\xi)=
      \begin{cases}
         \xi & \text{ if } m = n,\\
         \cA_{m,n} \xi + \dsum_{k=n}^{m-1} \cA_{m,k+1} P_{k+1}
            f_k(x_k(\xi),\phi_k(x_k(\xi))) & \text{ if } m > n.
      \end{cases}
   \end{equation}
   One can easily verify from~\eqref{cond-x_m-0}, \eqref{cond-phi-0} and
   \eqref{cond-f-0} that $(Jx)_m(0)=0$ for every $m \ge n$.

   Let $x \in \cB_n$ and let $\xi \in E_n$. From~\eqref{def:J},~\eqref{eq:lip_const},~\eqref{cond-phi-1}, ~\eqref{cond-x_m-1} and~\eqref{eq:alpha} it follows for every $m > n $ that
   \begin{align*}
      \|(Jx)_m (\xi)\|
      & \le \| \cA_{m,n}P_n\| \, \|\xi \|
         + \sum_{k=n}^{m-1} \|\cA_{m,k+1} P_{k+1}\| \, \|f_k(x_k(\xi),\phi_k(x_k(\xi)))\|\\
      & \le a_{m,n} \|\xi\| + \sum_{k=n}^{m-1} a_{m,k+1}  \, \Lip(f_k) \,
         \prts{\|x_k(\xi)\|+\|\phi_k(x_k(\xi))\|}\\
      & \le a_{m,n} \|\xi\| + \sum_{k=n}^{m-1} a_{m,k+1}  \, \Lip(f_k) \,
         2 \|x_k(\xi)\|\\
      & \le a_{m,n} \|\xi\| + 2 \sum_{k=n}^{m-1} a_{m,k+1}  \, \Lip(f_k) \,
         a_{k,n} \normn{x} \|\xi\|\\
      & \le a_{m,n} \|\xi\| + 2 \alpha \normn{x} a_{m,n} \|\xi\|\\
      & \le \prts{1 + 2 \alpha \normn{x}} a_{m,n} \|\xi\|
   \end{align*}
   and this implies
   \begin{equation}\label{ineq:abs(Jx)}
      \normn{Jx} \le 1 + 2 \alpha \normn{x}.
   \end{equation}
   Therefore we have the inclusion $J(\cB_n)\subset\cB_n$.

   We now show that $J$ is a contraction in $\cB_n$. Let $x,y\in \cB_n$. Then
   \begin{equation}\label{ineq:norm:J_x_m(xi)-J_y_m(bar_xi)}
      \begin{split}
         & \|(Jx)_m(\xi)-(Jy)_m(\xi)\|\\
         & \le \sum_{k=n}^{m-1} \|\cA_{m,k+1}P_{k+1}\| \
            \|f_k(x_k(\xi),\phi_k(x_k(\xi))) - f_k(y_k(\xi),\phi_k(y_k(\xi)))\|
      \end{split}
   \end{equation}
   for every $m \ge n$ and every $\xi \in E_n$. By~\eqref{eq:lip_const},~\eqref{cond-phi-1},~\ref{eq:split4} and~\eqref{cond-x_m-1} we have for every $k \ge n$
   \begin{equation}\label{ineq:norm:f_k(x_k)-f_k(y_k)}
      \begin{split}
         & \|f_k(x_k(\xi),\phi_k(x_k(\xi)))
            - f_k(y_k(\xi),\phi_k(y_k(\xi)))\|\\
         & \le \Lip(f_k) \prts{\| x_k(\xi) - y_k(\xi)\|
            + \|\phi_k(x_k(\xi)) - \phi_k(y_k(\xi))\|}\\
         & \le 2 \Lip(f_k) \| x_k(\xi) - y_k(\xi)\| \\
         & \le 2 \Lip(f_k) a_{k,n} \|\xi\| \ \normn{x-y}
      \end{split}
   \end{equation}
   Hence, from \eqref{ineq:norm:J_x_m(xi)-J_y_m(bar_xi)},~\ref{eq:split4}, \eqref{ineq:norm:f_k(x_k)-f_k(y_k)} and~\eqref{eq:alpha} we have
   \begin{align*}
      \|(Jx)_m(\xi) -(Jy)_m(\xi)\|
      & \le 2 \|\xi\| \ \normn{x-y}
         \sum_{k=n}^{m-1} a_{m,k+1} a_{k,n} \Lip(f_k)\\
      & \le 2 \alpha a_{m,n} \|\xi\| \ \normn{x-y}
   \end{align*}
   for every $m \ge n$ and every $\xi \in E_n$ and this implies
      $$ \normn{Jx-Jy} \le 2 \alpha \normn{x-y}.$$
   Since by~\eqref{ine:alpha_beta} we have $\alpha < 1/2$ it follows that $J$ is a contraction in $\cB_n$. Because $\cB_n$ is a Banach space by the Banach fixed point theorem, the map $J$ has a unique fixed point $x^\phi$ in~$\cB_n$, which is thus the desired sequence. Moreover, is obvious that~\eqref{cond-x^phi_n(xi)=xi} is true and by~\eqref{ineq:abs(Jx)} we have
      $$ \normn{x^\phi} \le 1 + 2 \alpha \normn{x^\phi}$$
   and since $\alpha < 1/2$ we have~\eqref{eq:norm_x^phi}.

   To prove~\eqref{ineq:x^phi_m(xi)- x^phi_m(barxi)<=...} we will first prove that for every $x \in \cB_n$, if
      $$ \|x_m(\xi)-x_m(\bar\xi)\|
         \le \dfrac{1}{1-2\alpha} a_{m,n} \|\xi-\bar\xi\|$$
   for every $m \ge n$ and every $\xi,\bar\xi \in E_n$, then
      $$ \|\prts{Jx}_m(\xi)-\prts{Jx}_m(\bar\xi)\|
         \le \dfrac{1}{1-2\alpha} a_{m,n} \|\xi-\bar\xi\|$$
   for every $m \ge n$ and every $\xi,\bar\xi \in E_n$. In fact
   \begin{align*}
      \|\prts{Jx}_m(\xi)-\prts{Jx}_m(\bar\xi)\|
      & \le \|\cA_{m,n}P_n\| \|\xi-\bar\xi\|
            + \dsum_{k=n}^{m-1} \|\cA_{m,k+1} P_{k+1}\| \gamma_k\\
      & \le a_{m,n} \|\xi-\bar\xi\| + \dsum_{k=n}^{m-1} a_{m,k+1} \gamma_k,
   \end{align*}
   where $\gamma_k = \|f_k(x_k(\xi),\phi_k(x_k(\xi))) - f_k(x_k(\bar\xi),\phi_k(x_k(\bar\xi)))\|.$ Since
   \begin{align*}
      \gamma_k
      & \le \Lip(f_k) \prts{\|x_k(\xi) - x_k(\bar\xi)\|
         + \|\phi_k(x_k(\xi))- \phi_k(x_k(\bar\xi))\|}\\
      & \le 2 \Lip(f_k) \|x_k(\xi) - x_k(\bar\xi)\| \\
      & \le \dfrac{2}{1-2\alpha} \Lip(f_k) a_{k,n} \|\xi-\bar\xi\|
   \end{align*}
   we have
   \begin{align*}
      \|\prts{Jx}_m(\xi)-\prts{Jx}_m(\bar\xi)\|
      & \le a_{m,n} \|\xi-\bar\xi\| + \dfrac{2}{1-2\alpha} \|\xi-\bar\xi\|
         \dsum_{k=n}^{m-1} a_{m,k+1} a_{k,n} \Lip(f_k)\\
      & \le a_{m,n} \|\xi-\bar\xi\|
         + \dfrac{2\alpha}{1-2\alpha} a_{m,n} \|\xi-\bar\xi\|\\
      & = \dfrac{1}{1-2\alpha} a_{m,n} \|\xi-\bar\xi\|.
   \end{align*}
   Now fix $z = \prts{z_m}_{m \ge n} = \prts{\cA_{m,n} P_n}_{m \ge n} \in B_n$. Since
      $$ \|z_m(\xi) - z_m(\bar\xi)\|
         \le a_{m,n} \|\xi-\bar\xi\|
         \le \dfrac{1}{1-2\alpha} a_{m,n} \|\xi-\bar\xi\|,$$
   we have
      $$ \|\prts{J^k z}_m(\xi) - \prts{J^k z}_m(\bar\xi)\|
         \le \dfrac{1}{1-2\alpha} a_{m,n} \|\xi-\bar\xi\|$$
   for every $k \in \N$. Letting $k \to +\infty$ in the last inequality we have~\eqref{ineq:x^phi_m(xi)- x^phi_m(barxi)<=...}.

   Next we will prove~\eqref{ineq:d_n(x^phi-x^psi)<=d(phi,psi)}. Let $\phi, \psi \in \cX$. From~\eqref{eq:dyn-split2a} we have
   \begin{equation}\label{ineq:norm:x_m^phi(xi)-x_m^psi(xi)}
      \begin{split}
         & \|x_m^\phi(\xi) - x_m^\psi(\xi)\|\\
         & \le \dsum_{k=n}^{m-1} \|\cA_{m,k+1} P_{k+1}\| \,
            \|f_k(x^\phi_k(\xi),\phi_k(x^\phi_k(\xi)))
            - f_k(x^\psi_k(\xi),\psi_k(x^\psi_k(\xi)))\|
      \end{split}
   \end{equation}
   for every $m \ge n$ and every $\xi \in E_n$. By~\eqref{eq:lip_const},~\eqref{cond-phi-1},~\eqref{cond-x_m-1}, ~\eqref{def:metric:X} and~\eqref{cond-x_m-1a} it follows that
   \begin{equation}\label{ineq:norm:f_k(x^phi_k...)- f_k(x^psi_k...)}
      \begin{split}
         & \|f_k(x^\phi_k(\xi),\phi_k(x^\phi_k(\xi)))
               - f_k(x^\psi_k(\xi),\psi_k(x^\psi_k(\xi)))\|\\
         & \le \Lip(f_k) \prts{\|x^\phi_k(\xi) - x^\psi_k(\xi)\|
            + \|\phi_k (x^\phi_k(\xi)) - \psi_k(x^\psi_k(\xi))\|}\\
         & \le \Lip(f_k) \prts{2 \|x^\phi_k(\xi) - x^\psi_k(\xi)\|
            + \|\phi_k (x^\psi_k(\xi)) - \psi_k(x^\psi_k(\xi))\|}\\
         & \le \Lip(f_k) \prtsr{2 a_{k,n} \|\xi\| \normn{x^\phi-x^\psi}
            + \|x^\psi_k(\xi)\| d(\phi,\psi)}\\
         & \le \Lip(f_k) a_{k,n} \|\xi\| \prtsr{2 \normn{x^\phi-x^\psi}
            + \dfrac{1}{1-2\alpha} d(\phi,\psi)}
      \end{split}
   \end{equation}
   for every $k \ge n$. Hence by~\eqref{ineq:norm:x_m^phi(xi)-x_m^psi(xi)}, \eqref{ineq:norm:f_k(x^phi_k...)- f_k(x^psi_k...)},~\ref{eq:split4} and~\eqref{eq:alpha} we get
   \begin{align*}
      \|x_m^\phi(\xi) - x_m^\psi(\xi)\|
      & \le \|\xi\| \prtsr{2 \normn{x^\phi-x^\psi}
         + \dfrac{1}{1-2\alpha} d(\phi,\psi)}
         \dsum_{k=n}^m a_{m,k+1}  a_{k,n} \Lip(f_k)\\
      & \le a_{m,n} \|\xi\| \prtsr{2 \alpha \normn{x^\phi - x^\psi}
         + \dfrac{\alpha}{1-2\alpha} d(\phi,\psi)}
   \end{align*}
   for every $m \ge n$ and every $\xi \in E_n$ and this implies
      $$ \normn{x^\phi-x^\psi}
         \le 2 \alpha \normn{x^\phi-x^\psi}
            + \dfrac{\alpha}{1-2\alpha} d(\phi,\psi).$$
   Therefore
      $$ \normn{x^\phi-x^\psi}
         \le \dfrac{\alpha}{\prts{1-2\alpha}^2} d(\phi,\psi).$$
\end{proof}

We now represent by $\prts{x_{n,k}^\phi}_{k \ge n} \in \cB_n$ the unique sequence given by Lemma~\ref{lemma:Exist-Suc-x_m}.

\begin{lemma} \label{lemma:equiv}
   Let $\phi \in \cX$. The following properties are equivalent:
   \begin{enumerate}[\lb=$\alph*)$,\lm=5mm]
      \item for every $n \in \N$, $m \ge n$ and $\xi \in E_n$ the identity~\eqref{eq:dyn-split2b} holds with $x_k = x_{n,k}^\phi$;
      \item for every $n \in \N$ and every $\xi \in E_n$
          \begin{equation} \label{eq:phi_n}
            \phi_n(\xi)
               = - \sum_{k=n}^\infty (\cA_{k+1,n}|_{F_n})^{-1}
                  Q_{k+1} f_k(x^\phi_{n,k}(\xi), \phi_k(x^\phi_{n,k}(\xi)))
          \end{equation}
          holds.
   \end{enumerate}
\end{lemma}

\begin{proof}
   First we prove that the series in~\eqref{eq:phi_n} is convergent.
   From~\ref{eq:split5},~\eqref{eq:lip_const},~\eqref{cond-phi-1a}, \eqref{cond-x_m-1a} and~\eqref{eq:beta}, we conclude that for every $n \in \N$ and every $\xi \in E_n$
   \begin{align*}
      & \sum_{k=n}^{\infty} \|(\cA_{k+1,n}|_{F_n})^{-1} Q_{k+1}
         f_k(x_{n,k}^{\phi}(\xi),\phi_k(x_{n,k}^{\phi}(\xi)))\|\\
      & \le \sum_{k=n}^{\infty} \|(\cA_{k+1,n}|_{F_n})^{-1} Q_{k+1}\| \
         \|f_k(x_{n,k}^{\phi}(\xi),\phi_k(x_{n,k}^{\phi}(\xi)))\| \\
      & \le \sum_{k=n}^{\infty} b_{k+1,n} \Lip(f_k)
         \prts{\|x^\phi_{n,k}(\xi)\| + \|\phi_k(x^\phi_{n,k}(\xi))\|}\\
      & \le 2 \|\xi\| \normn{x^\phi}
         \sum_{k=n}^{\infty} b_{k+1,n} a_{k,n} \Lip(f_k)\\
      & \le 2 \beta \|\xi\| \normn{x^\phi}
   \end{align*}
   and thus the series converges.

   Now, let us suppose that~\eqref{eq:dyn-split2b} holds with $x = x^\phi$ for every $n \in \N$, every $m \ge n$ and every $\xi \in E_n$. Then, since $(\cA_{m,n}|_{F_n})^{-1}\cA_{m,k+1}|_{F_{k+1}}=(\cA_{k+1,n}|_{F_n})^{-1}$ for $n \le k \le m-1$, equation~\eqref{eq:dyn-split2b} can be written in the following equivalent form\small
   \begin{equation} \label{eq:equiv1A} 
      \phi_{n}(\xi) = (\cA_{m,n}|_{F_n})^{-1} \phi_m(x_{n,m}^{\phi}(\xi)) -
      \sum_{k=n}^{m-1} (\cA_{k+1,n}|_{F_n})^{-1}
      Q_{k+1} f_k(x_{n,k}^{\phi}(\xi),\phi_k(x_{n,k}^{\phi}(\xi))).
   \end{equation}\normalsize
   Using~\ref{eq:split5},~\eqref{cond-phi-1a} and~\eqref{cond-x_m-1a},
   we have
   \begin{align*}
      \|(\cA_{m,n}|_{F_n})^{-1} \phi_m(x_{n,m}^{\phi}(\xi))\|
      & = \|(\cA_{m,n}|_{F_n})^{-1} Q_m \phi_m(x_{n,m}^\phi(\xi))\| \\
      & \le b_{m,n} \|x_{n,m}^\phi(\xi)\|\\
      & \le b_{m,n} a_{m,n} \|\xi\| \normn{x} \|\xi\|
   \end{align*}
   and by~\eqref{eq:CondicaoTeo} this converge to zero when $m\to\infty$.
   Hence, letting $m\to\infty$ in \eqref{eq:equiv1A} we obtain the
   identity~\eqref{eq:phi_n} for every $n \in \N$ and every $\xi \in E_n$.

   We now assume that for every $n \in \N$, $m \ge n$ and $\xi \in E_n$ the identity~\eqref{eq:phi_n} holds. Therefore
      $$ \cA_{m,n} \phi_n(\xi)
         = - \sum_{k=n}^{\infty} \cA_{m,n} (\cA_{k+1,n}|_{F_n})^{-1}
         Q_{k+1} f_k(x_{n,k}^{\phi}(\xi),\phi_k(x_{n,k}^{\phi}(\xi))),$$ and
   thus it follows from \eqref{eq:phi_n} and the uniqueness of the sequences
   $x^\phi$ that
   \begin{align*}
      \cA_{m,n} \phi_n(\xi) & + \sum_{k=n}^{m-1} \cA_{m,k+1}
         Q_{k+1} f_k(x_{n,k}^{\phi}(\xi),\phi_k(x_{n,k}^{\phi}(\xi))) \\
      & = - \sum_{k=m}^{\infty} (\cA_{k+1,m}|_{F_n})^{-1}
            Q_{k+1} f_k(x_{n,k}^{\phi}(\xi),\phi_k(x_{n,k}^{\phi}(\xi)))\\
      & = - \sum_{k=m}^{\infty} (\cA_{k+1,m}|_{F_n})^{-1}
            Q_{k+1} f_k(x_{m,k}^{\phi}(x_{n,m}^{\phi}(\xi)),
            \phi_k(x_{m,k}^{\phi}(x_{n,m}^{\phi}(\xi))))\\
      & = \phi_m(x_{n,m}^{\phi}(\xi))
   \end{align*}
   for every $n \in \N$, every $m \ge n$ and every $\xi \in E_n$. This proves the lemma.
\end{proof}

\begin{lemma} \label{lemma:Exist-Suc-phi}
   There is a unique $\phi \in \cX$ such that
      $$ \phi_n(\xi)
         = - \sum_{k=n}^\infty (\cA_{k+1,n}|_{F_n})^{-1}
            Q_{k+1} f_k(x^\phi_k(\xi), \phi_k(x^\phi_k(\xi)))$$
   for every $n \in \N$ and every $\xi \in E_n$.
\end{lemma}

\begin{proof}
   We consider the operator $\Phi$ defined for each $\phi\in \cX$ by
   \begin{equation} \label{eq:op-Phi}
      (\Phi \phi)_n (\xi) =
         - \dsum_{k=n}^{\infty} (\cA_{k+1,n}|_{F_n})^{-1} Q_{k+1} f_k
            (x_k^{\phi}(\xi), \phi_k(x_k^{\phi}(\xi)))
   \end{equation}
   where $x^{\phi}=(x^\phi_k)_{k\ge n} \in \cB_n$ is the unique sequence given by Lemma~\ref{lemma:Exist-Suc-x_m}. It follows from~\eqref{cond-f-0},~\eqref{cond-x_m-0},~\eqref{cond-phi-0} and~\eqref{eq:op-Phi} that $(\Phi\phi)_n(0)=0$ for each $n\in\N$.

Furthermore, given $n \in \N$ and $\xi, \bar\xi \in E_n$, by~\ref{eq:split5},~\eqref{eq:lip_const},~\eqref{cond-phi-1},~\eqref{ineq:x^phi_m(xi)- x^phi_m(barxi)<=...} and~\eqref{eq:beta} we have
   \begin{align*}
      & \|(\Phi \phi)_n (\xi) - (\Phi \phi)_n (\bar\xi)\| \\
      & \le \sum_{k=n}^{\infty} \|(\cA_{k+1,n}|_{F_n})^{-1} Q_{k+1}\|
         \cdot \|f_k(x^\phi_k(\xi),\phi_k(x^\phi_k(\xi)))
         - f_k(x^\phi_k(\bar\xi),\phi_k(x^\phi_k(\bar\xi)))\|\\
      & \le \sum_{k=n}^{\infty} b_{k+1,n} \ \Lip(f_k) \
         2 \|x^\phi_k(\xi) - x_k^\phi(\bar\xi)\|\\
      & \le \dfrac{2}{1-2\alpha} \|\xi-\bar\xi\|
         \sum_{k=n}^{\infty} b_{k+1,n} \ \Lip(f_k) \ a_{k,n} \\
      & \le \dfrac{2\beta}{1-2\alpha} \|\xi - \bar\xi\|
   \end{align*}
   Since by~\eqref{ine:alpha_beta} the inequality $\alpha + \beta < 1/2$ holds, it follows that
      $$ \|(\Phi \phi)_n (\xi) - (\Phi \phi)_n (\bar\xi)\|
         \le \|\xi - \bar\xi\|.$$
   Therefore $\Phi(\cX)\subset\cX$.

   We now show that $\Phi$ is a contraction. Given $\phi,\psi\in\cX$
   and $n\in \N$, let $x^{\phi}$ and $x^{\psi}$ be the unique sequences given
   by Lemma~\ref{lemma:Exist-Suc-x_m} respectively for $\phi$ and $\psi$.
   By~\ref{eq:split5},~\eqref{ineq:norm:f_k(x^phi_k...)-
   f_k(x^psi_k...)},~\eqref{ineq:d_n(x^phi-x^psi)<=d(phi,psi)}
   and~\eqref{eq:beta} we have
   \begin{align*}
      & \|(\Phi \phi)_n (\xi)-(\Phi \psi)_n(\xi)\| \\
      & \le \sum_{k=n}^{\infty} \|(\cA_{k+1,n}|_{F_n})^{-1} Q_{k+1}\|
         \|f_k(x^\phi_k(\xi),\phi_k(x^\phi_k(\xi)))
         - f_k(x^\psi_k(\xi),\phi_k(x^\psi_k(\xi)))\|\\
      & \le \sum_{k=n}^{\infty} b_{k+1,n} \Lip(f_k) a_{k,n} \|\xi\|
            \prtsr{2 \normn{x^\phi-x^\psi}
            + \dfrac{1}{1-2\alpha} d(\phi,\psi)}\\
      & \le \sum_{k=n}^{\infty} b_{k+1,n} \Lip(f_k) a_{k,n} \|\xi\|
            \prtsr{\dfrac{2\alpha}{\prts{1-2\alpha}^2} + \dfrac{1}{1-2\alpha}} d(\phi,\psi)\\
      & \le \dfrac{1}{\prts{1-2\alpha}^2} \|\xi\| d(\phi,\psi)
         \sum_{k=n}^{\infty} b_{k+1,n} a_{k,n} \Lip(f_k)\\
      & \le \dfrac{\beta}{\prts{1-2\alpha}^2}
         \|\xi\| d(\phi,\psi)
   \end{align*}
   for every $n \in \N$ and every $\xi \in E_n$ and this implies
      $$ d(\Phi \phi, \Phi \psi)
         \le \dfrac{\beta}{\prts{1-2\alpha}^2}  d(\phi,\psi)$$
   Since by~\eqref{ine:alpha_beta} we have $\dfrac{\beta}{\prts{1-2\alpha}^2} < 1$, it follows that $\Phi$ is a contraction in $\cX$. Therefore the map $\Phi$ has a unique fixed point $\phi$ in~$\cX$ that is the desired sequence.
\end{proof}

We are finally in conditions to prove Theorem~\ref{thm:global}.

   By Lemma~\ref{lemma:Exist-Suc-x_m}, for each $\phi\in \cX$ there is
   a unique sequence $x^{\phi}\in\cB_n$ satisfying~\eqref{eq:dyn-split2a}. It
   remains to show that there is a $\phi$ and a corresponding $x^\phi$ that
   satisfie~\eqref{eq:dyn-split2b}. By Lemma~\ref{lemma:equiv}, this is
   equivalent to solve~\eqref{eq:phi_n}. Finally, by
   Lemma~\ref{lemma:Exist-Suc-phi}, there is a unique solution of
   \eqref{eq:phi_n}. This establishes the existence of the invariant manifolds. Moreover, for each $n\in \N$, $m\ge n$ and
   $\xi,\bar{\xi}\in E_n$ it follows from~\eqref{cond-phi-1} that
   \begin{align*}
      & \|\cF_{m,n}(\xi,\phi_n(\xi)) - \cF_{m,n}(\xi,\phi_n(\bar\xi))\|\\
      & \le \|x_m(\xi) - x_m(\bar\xi)\|
         + \|\phi_m(x_m(\xi)) - \phi_m(x_m(\bar\xi))\|\\
      & \le 2 \|x_m(\xi) - x_m(\bar\xi)\|\\
      & \le \dfrac{2}{1-2\alpha} a_{m,n} \|\xi-\bar \xi\|.
   \end{align*}
   Hence we obtain~\eqref{thm:ineq:norm:F_mn(xi...)-F_mn(barxi...)} and the
   theorem is proved.
\section{Proof of Theorem~\ref{thm:local}}
We will now establish Theorem~\ref{thm:local}.
Since, for each $k \in \N$, $f_k|_{B(r_k)} \colon B(r_k) \to X$ is a Lipschitz function, $f_k$ can be continuously  extended in unique way to the closure of $B(r_k)$ and it is easy to see that the function $\tilde{f}_k \colon X \to X$ given by
\[
\tilde{f}_k =
\begin{cases}
f_k(x) & \text{if} \quad x \in B(r_k) \\
f_k \left( x r_k / \|x\| \right) & \text{if} \quad x \notin B(r_k)
\end{cases}
\]
is Lipschitz with $\Lip(\tilde{f}_k) \le 2\Lip(f_k|_{B(r_k)})$. Thus we have
\[
\tilde{\alpha} = \sup_{\underset{\scriptstyle m \ne n}{(m,n) \in \N_\ge^2}}
\dfrac{1}{a_{m,n}} \dsum_{k=n}^{m-1} a_{m,k+1} a_{k,n} \Lip(\tilde{f_k})
\le 2 \alpha < + \infty
\]
and
\[
\tilde{\beta} = \sup_{n \in \N} \dsum_{k=n}^{+\infty} b_{k+1,n} a_{k,n} \Lip(\tilde{f_k}) \le 2 \beta < + \infty.
\]
According to~\eqref{ine:alpha_beta-local} we have
$$2\tilde{\alpha} + \max\set{2\tilde{\beta}, \sqrt{\tilde{\beta}}}
\le 4 \alpha + \max\set{4 \beta, \sqrt{2\beta}} < 1,$$
and thus Theorem~\ref{thm:global} shows that~\eqref{thm:global:invar} and~\eqref{thm:ineq:norm:F_mn(xi...)-F_mn(barxi...)} hold for the perturbations $\tilde{f_k}$.

Let $\tilde{\cF}(m,n)$ be the operators in~\eqref{eq:dyn} with the functions $f_k$ replaced by the functions $\tilde{f}_k$.
From~\eqref{thm:ineq:norm:F_mn(xi...)-F_mn(barxi...)} we obtain
\begin{equation} \label{eq:aux-local}
\begin{aligned}
\|\tilde{\cF}(m,n)(\xi,\phi_n(\xi))-  \tilde{\cF}(m,n)(\bar\xi,\phi_n(\bar\xi))\|
& \le
\frac{2}{1-2\tilde{\alpha}} \, a_{m,n} \|\xi-\bar\xi\| \\
& \le
 \frac{2}{1-4 \alpha} \, a_{m,n} \|\xi-\bar\xi\|.
\end{aligned}
\end{equation}
for every $\prts{m, n} \in \N^2_\ge$ and every $\xi$, $\bar\xi \in E_n$.
In particular, if $\bar\xi=0$ and $\xi \in B(r_n/(2s_n)) \cap E_n$, we have $(\xi,\phi_n(\xi)) \in B(r_n/s_n)$ and by~\eqref{ine:sn-local} we get
   $$ \| \tilde{\cF}(m,n)(\xi,\phi_n(\xi))\|
      \le \frac{2}{1-4 \alpha} \, a_{m,n} \|\xi\|
      < \frac{2}{1-4\alpha} \, a_{m,n} \, \frac{r_n}{s_n} \le r_m$$
and thus $\tilde{\cF}(m,n) \left(\cV^*_{\phi,n,r_n/(2s_n)} \right) \subseteq \cV^*_{\phi,m,r_m}$, where $\cV_{\phi,n}$ in~\eqref{eq:manifold-local} correspond to the manifolds obtained in Theorem~\ref{thm:global} for the perturbations $\tilde{f}_k$. Since $\tilde{f}_k|_{B(r_k)}=f_k|_{B(r_k)}$ for each $k \in \N$, we obtain
$$\cF(m,n) \left(\cV^*_{\phi,n,r_n/(2s_n)} \right) \subseteq \cV^*_{\phi,m,r_m}.$$
and~\eqref{thm:global:invar-local} holds. Finally, using~\eqref{eq:aux-local} and still recalling that $\tilde{f}_k=f_k$ in $B(r_k)$ for each $k \in \N$, we obtain~\eqref{thm:ineq:norm:F_mn(xi...)-F_mn(barxi...)-local}.
\section*{Acknowledgments}
This work was partially supported by FCT though Centro de Matemática da Universidade da Beira Interior (project PEst-OE/MAT/UI0212/2011).
\bibliographystyle{elsart-num-sort}

\end{document}